\documentclass[12pt]{article}

\usepackage{amsfonts}
\usepackage{amsmath}
\usepackage{amsthm}
\usepackage{url}

\newtheorem{theorem}{Theorem}
\newtheorem{proposition}[theorem]{Proposition}
\newtheorem*{theorem*}{Theorem}
\newtheorem*{lemma*}{Lemma}
\newtheorem{lemma}[theorem]{Lemma}
\newtheorem{conj}{Conjecture}

\newtheorem{question}[conj]{Question}

\newtheorem*{corollary*}{Corollary}

\newtheorem*{defn}{Definition}
\theoremstyle{remark}
\newtheorem*{remark}{Remark}

\def\eps{\varepsilon}

\newcommand\Z{\mathbb{Z}}
\newcommand\N{\mathbb{N}}
\newcommand\R{\mathbb{R}}
\newcommand\Q{\mathbb{Q}}
\newcommand\bprime{{\hat{b}}}

\newcommand\szt{{\rm{SzT}}}

\begin{document}
\title{An $n$-in-a-row type game}
\author{Joshua Erde\and Mark Walters}
\date\today
\maketitle
\begin{abstract} 
  We consider a Maker-Breaker type game on the plane, in which each
  player takes $t$ points on their $t^\textrm{th}$ turn.  Maker wins
  if he obtains $n$ points on a line (in any direction) without any of
  Breaker's points between them. We show that, despite Maker's
  apparent advantage, Breaker can prevent Maker from winning until
  about his $n^\textrm{th}$ turn.  We actually prove a stronger
  result: that Breaker only needs to play $\omega(\log t)$ points on
  his $t^\textrm{th}$ turn to prevent Maker from winning until this
  time.

  We also consider the situation when the number of points claimed by
  Maker grows at other speeds, in particular, when Maker claims
  $t^\alpha$ points on his $t^\textrm{th}$ turn.
\end{abstract}

\section{Introduction}
The ordinary \emph{$n$-in-a-row game} is a Maker-Breaker game played
on $\Z^2$, where two players, Maker and Breaker, take turns claiming
unclaimed points in the plane.  Maker wins if he can claim $n$
consecutive points in a row either vertically, horizontally, or
diagonally, otherwise Breaker wins. It is known that for $n \leq 4$
this game is a Maker win, and for $n \geq 8$ the game is a Breaker win
\cite{Z1980}. It will be convenient to consider the
  game as consisting of a series of timesteps, each of which consists
  of one of Maker's turns and the subsequent turn of Breaker.

Erde \cite{E2012} considered a variation of this game where, 
  instead of picking one point each turn, the
  number of points picked by a player on his $t$\textsuperscript{th}
  turn is a function of $t$. In particular, suppose on
   their first turn Maker 
  and Breaker each claim $1$ point, but on their second turn they each claim $2$ 
  points, and $3$ on their third, and so on. 
  Unlike the $n$-in-a-row game this game is clearly never a Breaker win, since on his
  $n$\textsuperscript{th} turn Maker can claim
  an entire winning line.  However, Erde showed that Maker cannot win
  this game in time less than $(1-o(1))n$ (i.e., before the
  $(1-o(1))n$\textsuperscript{th} timestep).
 
  One generalization of the $n$-in-a-row game is to allow the winning
  lines to have arbitrary slopes; that is we play a Maker-Breaker game
  on $\Z^2$ where Maker wins if he can claim $n$ consecutive points in
  a line, with \emph{any} slope. Clearly this game is easier for Maker
  than the $n$-in-a-row game. Indeed Beck~\cite{B2008} showed that
  this game is a Maker win for all $n$ (recall the ordinary game is
  known to be a Breaker win for all $n\ge 8$).  Since this game is
  easier for Maker it raises the possibility that the analogous
  modified version where the number of points picked on each turn is
  increasing can be won by Maker in time less than $(1-o(1))n$.
 
    In fact, we consider an even easier game for Maker. The players
    take it in turns to claim points in $\Z^2$, with each player
    claiming $t$ points on their $t^\textrm{th}$ turn. Maker wins if
    he gets $n$ points on a straight line (in any direction at all)
    with no point of Breaker's between the first and last of these
    $n$. The points need not be consecutive and there may be other
    points of Maker or Breaker on this line. (Note this is not a
    standard Maker-Breaker game as Maker's winning sets depend on
    Breaker's points.)  We call such a line segment a \emph{winning
      line segment}.

  As before Maker can claim the whole of a winning line segment on his
  $n^\textrm{th}$ turn and so he can definitely win by time $n$. Our
  first result is that, even in this game, this is essentially the
  best Maker can do.

\begin{theorem}\label{t:basic-critical}
  In the above game Breaker can stop Maker winning before time $(1-o(1))n$.
\end{theorem}

\begin{remark}Throughout the paper we use the standard notations $O,o$
  and $\Omega$ as well as the common, but less standard, notation
  $f=\omega(g)$ which denotes the property that $\lim_{n\to \infty}
  f(n)/g(n)=\infty$.
\end{remark}

We also extend Theorem~\ref{t:basic-critical} substantially.  Fix
functions $m,b\colon\N\to\N$ and suppose that on his $t^\textrm{th}$
turn Maker plays $m(t)$ points and Breaker plays $b(t)$ points in his. For
simplicity we shall assume that both $m(t)$ and $b(t)$ are monotone
increasing. Roughly, we say Maker wins if he can get
  $n$ points in a row significantly before the time at which he is playing
  $n$ points in a single turn. The precise definition is the
  following.

\begin{defn}
  Define $\tau_n$ to be the earliest time by which Maker can guarantee
  to have formed a winning line segment of $n$ points.  We say the
  game is a Breaker win if $m(\tau_n)=(1-o(1))n$, and we say it is a
  Maker win (with constant $\eps$) if there exists $\eps>0$ such that
  $m(\tau_n)<(1-\eps)n$ for all sufficiently large $n$.
\end{defn}

\noindent%
We remark that for some functions $b$ and $m$ neither of the above
will hold.

We wish to prove bounds on $m$ and $b$ showing when the game is a
Maker win and when it a Breaker win. We concentrate on the case when
$m(t)=t^\alpha$. For $\alpha \geq 1$ we have a rather surprising result
which is essentially tight. 
\begin{theorem}\label{main-theorem-a>1}
  Suppose that $m(t)=t^\alpha$ for $\alpha \geq 1$. Then if $b(t)=
  O(\log t)$ the game is a Maker win whereas if $b(t) = \omega(\log t)$
  the game is a Breaker win.
\end{theorem} 

For $\alpha<1$ we have only a weaker upper bound.
\begin{theorem}\label{t:subcritical-upperbound}
  Suppose that $m(t)=t^\alpha$ for some $\alpha<1$. Then if $b(t)=
  \omega(t^{1-\alpha})$ the game is a Breaker win.
\end{theorem}

Despite the much weaker upper bound we have no better lower bound
for~$b$ than in Theorem~\ref{main-theorem-a>1}; that is, the best we
can say is that Maker has a winning strategy if $b(t)= O( \log t)$.
However, our strategy for Breaker in both of the theorems
has a special form: Breaker never relies on placing a single point on
two of Maker's lines. More precisely Breaker can still win in all the
above cases if whenever he places any point he also has to designate a
specific direction and it only breaks Maker's winning sets through
that point in that specific direction.  (Breaker may play the same
point with a different direction but that counts as an extra point.)
We call the version of the game where Breaker has to specify this
direction the \emph{directed-Breaker} game.  We can prove that
Theorem~\ref{t:subcritical-upperbound} is essentially tight
for this game.
\begin{theorem}\label{t:subcritical-lowerbound}
  Suppose that $m(t)=t^\alpha$ for $\alpha<1$ and
  $b(t)=o(t^{1-\alpha})$. Then Maker wins the directed-Breaker
    game.
\end{theorem}

\section{Elementary remarks}\label{s:elem}
We start with a trivial remark. We can think of the points that
Breaker has claimed at any time in the game as having split each line
in the plane into a number of \emph{line segments}: blocks of points
that are unclaimed or claimed by Maker, either lying in between two of
Breaker's points, or one of Breaker's points and infinity (i.e., an
infinite ray), or a line not containing any of Breaker's points.  We
call any such segments that do not contain at least $n$ integer points
in total \emph{inactive}. Inactive line segments are not useful to
Maker as they cannot be extended to a winning line segment.

Suppose that Maker wins the game in time $T$ with $m(T)\le
(1-\eps)n$. Then after Breaker's last turn there must have been an
active line segment containing at least $\eps n$ of Maker's
points. Thus, we see that it is important for Breaker to try to ensure
that no such line segments exist at the end of his turn.

\begin{defn}
  Suppose that $l$ is a line segment containing some of Maker's
  points. We say that Breaker has \emph{$\eps$-split} the line $l$ if
  he has placed points on it splitting it into smaller line segments
  such that, after the split, there are no active segments containing
  (at least) $\eps n$ of Maker's points. In cases where $\eps$ is clear
  we will just say that Breaker has \emph{split} the line.
\end{defn}

The following lemma provides a simple bound on how many points Breaker
needs to split a line.
\begin{lemma}\label{l:split}
  Suppose that $\eps>0$ and that $l$ is an active line segment
  containing less than $n$ of Maker's points. Then Breaker can $\eps$-split
  the line $l$ using at most $2/\eps$ points.
\end{lemma}
\begin{proof}
  Breaker starts from one end and counts along  $\eps n-1$ of Maker's
  points. He would like to play the next integer point, say $x$, on
  the line but this may not be possible as $x$ may already be a Maker
  point. So instead he plays both the last free integer point before
  $x$ on the line segment and the first free integer point after
  $x$. He then repeats the process on the remaining line segment after
  the second of these points. Obviously when this process has finished
  there is no active line segment containing $\eps n$ of Maker's
  points and Breaker has played at most $2/\eps$ points.
\end{proof}

Next we describe a simpler related game that we will mention at
times. In the game as described so far there are advantages and
disadvantages to playing first: if Breaker claims a point then it stops
his opponent from claiming it (an advantage) but means Maker knows
where he has played (a disadvantage). 

In the modified version
Maker chooses some number of timesteps $T$ and $\eps>0$ such that
$m(T)=(1-\eps)n$, and then he gets to play all the points he would have
played up to that point in the original game at once. Then Breaker
plays all of his, with the additional freedom that he may choose points that Maker has
already chosen. Breaker wins if Maker's largest active line segment of
size at the end of the game has size less than $\eps n$. This game is
easier for Breaker than the standard game -- it is clear that if
Breaker can win the standard game then he can also win this game --
and it is easier to think about.  We will refer to this game as the
\emph{batched} game.

The key tool in our proofs is the Szemer\'edi-Trotter
Theorem~\cite{MR729791}. 
\begin{theorem*}[Szemer\'edi-Trotter]
  Suppose $P$ is a set of points, $L$ is a set of line segments in $\R^2$ 
  such that any two line segments meet in at most one point,
  and let $I$ be the set of incidences (an incidence is a point-line segment 
  pair with the line segment containing that point). Then
  \[
  |I|=O\left(|P|^{2/3}|L|^{2/3}+|P|+|L|\right)
  \]
\end{theorem*}
\noindent%
and its corollary
\begin{corollary*}[Szemer\'edi-Trotter]
  Let $P$, $L$ be as above. Then the number of
  line segments containing at least $k$ points is
  \[
  O\left(\frac{|P|^2}{k^3}+\frac{|P|}{k} \right)
  \]
\end{corollary*}
We remark that this theorem is normally stated in terms of lines
rather than line segments, however it is a folklore result that the
line segment version is implied as a simple consequence.

As we will be using the Szemer\'edi-Trotter Theorem it is convenient
to make the following definition.
\begin{defn}
  For any $n$ and $k$ define $SzT(n,k)$ to be the maximum number of
  incidences that can occur between $n$ points and $k$
  (non-overlapping) line segments.
\end{defn}

We start by illustrating our methods by applying them to the
batched game.
\begin{proposition}\label{p:batched}
  Suppose $m(t)=t$ and $b(t)=\omega(1)$. Then Breaker can win the
  batched game.
\end{proposition}
\begin{proof}
We aim for a contradiction: suppose Maker can win with
$T=(1-\eps)n$. Since $m(t)=t$ the total number of points Maker plays up
to the $T$\textsuperscript{th} timestep is less than $Tn$. By the corollary to
the Szemer\'edi-Trotter Theorem the number of lines segments Maker can
create containing more than $\eps n$ points is at most
\[
O\left(\frac{(Tn)^2}{(\eps n)^3}+\frac{Tn}{\eps n} \right)=O(T).
\]

Some of these line segments might contain significantly more than
$\eps n$ points, and would therefore take more than $O(1)$ points to
$\eps / 2$-split. However, we can count a line segment containing
$\sim k\eps n$ points as $\lceil k \rceil$ line segments each
containing $\leq \eps n$ points and since we are using the line
segment version of the Szemer\'edi-Trotter Theorem this doesn't change
our bound.  

Since $b(t)=\omega(1)$, Breaker has $\omega(T)$ points to
play. Therefore, by (the ideas of) Lemma~\ref{l:split}, Breaker can
$\eps$-split each of these segments. This guarantees that there are no
segments left with length greater than $\eps n$ and, thus, Breaker
wins.
\end{proof}
\section{A weighted bin game}

In this section we introduce a weighted bin game. This is a simpler
game that is easy to analyse but our main proofs will compare the real
game with this simple game. It is a single player (who we will call
Maker) game.
  
\begin{defn}
  Suppose $T$ is a constant and $b,M\colon \N\to \N$ are
  functions. The \emph{weighted bin game $(b,M,T)$} is the following
  game. The game is played with $1+\sum_{t=1}^T b(t)$ bins. On turn
  $t$ Maker places adds some weight to some bins subject to some
  constraints given below. Then the $b(t)$ largest bins are
  killed. The game lasts $T$ turns after which there is a single
  remaining live bin. Maker's aim is to maximise the weight of this
  remaining live bin.

  The constraint for Maker is the following: for any $s>0$ the total
  weight added during the last $s$ turns is at most $M(s)$.
\end{defn}

\begin{lemma}
  Suppose that Maker plays the game above playing weight $w_i$ on his
  $i^\textrm{th}$ turn for each $i$. Then the weight remaining in the
  last bin is at most
    \[
    \sum_{s=1}^T\frac{w_s}{\sum_{t=s}^Tb(t)+1}.
    \] 
\end{lemma}
\begin{proof}
  Suppose at the start of his $s$\textsuperscript{th} turn the
  $\sum_{t=s}^Tb(t)+1$ remaining live bins have average weight $a$.
  Clearly, after Maker adds this turn's weight the average weight is
  \[a+\frac{w_s}{\sum_{t=s}^Tb(t)+1}.\] and this average weight
  does not increase when the $b(t)$ largest bins are killed. 

  Hence, at the end of the game there is a single bin with (average)
  weight at most
  \[
  \sum_{s=1}^T\frac{w_s}{\sum_{t=s}^Tb(t)+1}
  \] 
  as claimed.
\end{proof}
\begin{remark}
  Obviously Maker can obtain the bound given in this lemma, but we
  shall not make use of that.
\end{remark}
\begin{lemma}\label{l:solo-bin-ball}
  Suppose that $T$ is fixed, that $M(s)$ is any increasing function
  $\N\to\N$ with $M(0)=0$, and that $b(t)$ has the property that, for
  any $s$, we have $\sum_{t=s}^Tb(t)\ge b(T)(T-s+1)/2$.  Let $\Delta
  M$ be the function given by $\Delta M(s)=M(s)-M(s-1)$. Then any
  strategy for the weighted bin game $(b,M,T)$ finishes with at most
  weight
  \[\frac{2}{b(T)}\sum_{t=1}^T\frac{\Delta M(t)}{t}
  \]
  in the final bin.
\end{lemma}
\begin{remark}
  The constraint on $b$ is roughly requiring that $b$ not be
  super-linear in growth.
\end{remark}
\begin{proof}
  Suppose Maker plays weight $w_t$ on the $t^\textrm{th}$ turn.
  By the previous lemma the weight in the final bin is at most
  \[
  \sum_{s=1}^T \frac{w_s}{\sum_{t=s}^Tb(t)+1}.\tag{$*$}
  \] 
  The constraint on Maker is that, for any $s$,
  \[
  \sum_{t=s}^Tw_t\le M(T-s+1).
  \]
  To maximise the bound $(*)$ we push weight onto the $w_t$ with $t$
  large. Thus the weight of the final bin is at most what Maker would
  get if he played weight $\Delta M(s)$ on step $T-s+1$.

  To conclude the proof we just have to bound $(*)$ in this case.
  We have
  \[
  \sum_{s=1}^T \frac{w_s}{\sum_{t=s}^Tb(t)+1}
  \le \sum_{s=1}^T \frac{\Delta M(T-s+1)}{\frac{1}{2}b(T)(T-s+1)+1}
  \le\frac{2}{b(T)}\sum_{t=1}^T\frac{\Delta M(t)}{t}
  \] as claimed.
\end{proof}

\subsection*{From the weighted bin game to the real game}
In this section we show why the weighted bin game is relevant: we show
that if Maker can win the real game (with certain parameters) then he
can obtain a certain weight in the weighted game (with certain
related parameters). Since we know exactly what weight Maker can
obtain in the weighted game this enables us to deduce results about
the real game.
\begin{lemma}\label{l:bin-ball-to-real}
  Suppose that $m(t)=t^\alpha$ and $b(t)$ are such that the
  $n$-in-a-row game is a Maker win with constant $\eps$.  Let
  $\bprime=\bprime(t)$ be the function $\frac{\eps}{4}b(t)$. Then for
  all sufficiently large $n$ there exists $T=T(n)$, with $T^\alpha\le
  n$, such that there is a Maker strategy for the
  weighted bin game $(\bprime,M, T)$, where
  $M(s)=\szt(T^{\alpha}s,\bprime(T) s+1)$ giving at least weight $\eps
  n/2$ in the last remaining bin.
\end{lemma}
\begin{proof}
  Consider the following Breaker strategy.  As in the statement of
  the lemma let $\bprime=\frac\eps 4b$. On his turn Breaker $\eps /
  2$-splits the $\bprime$ active line segments containing the most
  points of Maker's.  Breaker can do this since, assuming Maker has
  not already won, there was no line segment with more than $n$ points
  at the end of Makers turn. Thus, by Lemma~\ref{l:split} these
  $\bprime$ line segments can be split using at most $b$ points.

  Since Maker has a winning strategy it must win against this
  particular Breaker strategy. Suppose Maker wins in some number of
  turns $T+1$. We show that this $T$ satisfies the conclusion of the
  theorem.  Trivially, from the definition of a win, we have $T^{\alpha}
  \le (T+1)^{\alpha} \le n$.  Note that, at the end of Breaker's $T$\textsuperscript{th} turn
  Maker must have had an active line segment containing at least $\eps
  n$ points.

  Let $M$ be as in the statement of the theorem and define
  $B(s)=\sum_{t=T-s+1}^T\bprime(t)$.

  Consider the $B(T)+1$ line segments given by the $B(T)$ segments
  that Breaker splits during the game together with Maker's winning
  line segment.

  We create a bin $Q_i$ corresponding to each line segment $L_i$.  We
  map a situation in the point-line game to a situation in the
  weighted bin game as follows: place weight $l_i-\eps n/2$ in bin
  $Q_i$ when $l_i$ is the number of points on line $L_i$, placing
  weight zero if this is negative. 

  Note we do not consider the line segment as being `created' until
  Breaker has played it's endpoints: in particular if Maker plays some
  points on a line segment it only adds weight to the bin
  corresponding to the current segment, not sub-segments that will be
  created later. This does not matter because, by the definition of
  splitting a line, all the newly created line segments have at most
  $\eps n/2$ points so map to empty bins. Thus Breaker's move under
  this strategy corresponds to killing the $\bprime$ heaviest bins the
  weighted bin game.

  After the $T$\textsuperscript{th} timestep in the point-line game the remaining live line
  segment has at least $\eps n$ points so the corresponding bin has
  at least weight $\eps n/2$.
  
  Finally, we check the bound on $M$. Consider the last $s$ timesteps in
  the point-line game.  In these turns the weight added in the
  weighted bin game is (at most) the number of new incidences in the
  original game.  The total number of points added in these turns is
  at most $T^{\alpha}s$ and there are $B(s)+1\le \bprime(T) s+1$
  remaining live line segments. Since we are only counting the
  \emph{current} line segments these segments only meet in single
  points, so the Szemer\'edi-Trotter Theorem does apply.

  The number of points added to these $\bprime(T) s+1$ lines is at
  most the number of points on the largest $\bprime(T) s+1$ lines
  through these (at most) $T^{\alpha}s$ points. By definition this is
  at most $\szt(T^{\alpha}s,\bprime(T) s+1)$ which concludes the proof.
\end{proof}

Next we combine Lemma~\ref{l:bin-ball-to-real} with our bounds for when Maker can win 
the weighted bin game to give an explicit bound on how quickly $b(t)$ can grow, 
if Maker wins the $n$-in-a-row game. 
\begin{lemma}\label{c:upper}
  Suppose that $m(t)=t^\alpha$ and $b(t)$ are such that the
  $n$-in-a-row game is a Maker win with constant $\eps$. Further
  suppose that, for all sufficiently large $t$ and any $s$, we have
  $\sum_{i=s}^{t}b(i)\ge b(t)(t-s+1)/2$.  Then there exists $T=T(n)$
  with $T^\alpha\le n$ such that
\[
\frac{1}{b(T)}\left(T^{(2\alpha+1)/3}b(T)^{2/3}+ T^\alpha\log T+b(T)\log T\right)=\Omega(n).
\]
\end{lemma}
\begin{proof}
  By Lemma~\ref{l:bin-ball-to-real} Maker has a strategy for a
  weighted bin game $(\bprime,M,T)$ where $\bprime, M$ and $T$ are as
  in that lemma finishing with at least weight $\eps n/2$ in the last
  remaining bin.  By the Szemer\'edi-Trotter Theorem we have
\[
M(s)=\szt(T^\alpha s, \bprime(T) s +1)
  \le C'\left((T^\alpha s)^{2/3}(\bprime(T)s+1)^{2/3}+T^\alpha s+\bprime(T)s+1\right)
\]
Let $M'(s)$ be the right hand side of this equation. Since $M'\ge M$
Maker can also play the weighted bin game $(\bprime,M',T)$ finishing
with at least weight $\eps n/2$ in the last remaining bin. Now
\begin{align*}
\Delta M'(s)&=C'(T^\alpha)^{2/3} \left( \left((s+1)(\bprime(T)(s+1)+1)\right)^{2/3} - \left(s(\bprime(T)s +1)\right)^{2/3} \right)
\\& \qquad\qquad +C'T^\alpha+C'\bprime(T)\\
&=O((T^\alpha\bprime(T))^{2/3}s^{1/3}+T^\alpha+\bprime(T)).
\end{align*}
Thus, Lemma~\ref{l:solo-bin-ball} implies that
\begin{align*}
\eps n/2&=O\left( \frac{1}{\bprime(T)}\sum_{t=1}^T\frac{(T^\alpha\bprime(T))^{2/3}t^{1/3}+T^\alpha+\bprime(T)}{t}\right)\\
&=O\left( \frac{1}{\bprime(T)}\left(T^{(2\alpha+1)/3}\bprime(T)^{2/3}+ T^\alpha\log T+\bprime(T)\log T\right)\right).
\end{align*}
Since $\bprime=\frac{\eps}{4}b$ and $\eps$ is a constant this rearranges to give the result.
\end{proof}

\section{Upper bounds} 

\begin{proof}[Proof of Theorem~\ref{main-theorem-a>1}]
  Suppose $\alpha\ge 1$ and $b(t)=\omega(\log t)$ and the game is a
  Maker win with constant $\eps$. Since replacing $b$ by a smaller
  function which is also $\omega(\log t)$ only makes things harder for
  Breaker, we may additionally assume that, for all sufficiently large
  $t$ and any $s$, we have $\sum_{i=s}^{t}b(i)\ge b(t)(t-s+1)/2$. Thus
  by Lemma~\ref{c:upper} there exists some $T$, with $T^{\alpha} \leq n$, such that
\[
\frac{1}{b(T)}\left(T^{(2\alpha+1)/3}b(T)^{2/3}+ T^\alpha\log T+b(T)\log T\right)=\Omega(n).
\]
However, this is a contradiction, since $T^{(2\alpha+1)/3}\le T^\alpha \le n$ and $b(T) = \omega(\log T)$.

Thus, for any $\eps>0$ the game is not a a Maker win with constant
$\eps$ for any sufficiently large $n$, and thus the game is a Breaker
win.
\end{proof}

\begin{proof}[Proof of Theorem~\ref{t:subcritical-upperbound}]
  Suppose $\alpha<1$ and $b(t)=\omega(t^{1-\alpha})$ and the game is a
  Maker win with constant $\eps$. As in the previous proof we may
  additionally assume that, for all sufficiently large $t$ and any
  $s$, we have $\sum_{i=s}^{t}b(i)\ge b(t)(t-s+1)/2$.  Again,
  Lemma~\ref{c:upper} implies that there exists some $T$, with $T^{\alpha} \leq n$, such that 
\begin{align*}
\frac{1}{b(T)}\left(T^{(2\alpha+1)/3}b(T)^{2/3}+ T^\alpha\log T+b(T)\log T\right)
=\Omega(n).
\end{align*}
As before, since $T^\alpha\le n$ and $b(T)= \omega(T^{1-\alpha})$,
this is a contradiction. So, as in the proof of
Theorem~\ref{main-theorem-a>1}, the game is a Breaker win.
\end{proof}

\section{Lower bounds}\label{s:lower-bounds}

To complete the proof of Theorem~\ref{main-theorem-a>1} we need to 
prove that if $m(t)=t^\alpha$ and $b(t) = O( \log t$) then the game is a Maker
win.

The rough idea is that Maker can follow the (implicit) strategy given
for the weighted bin game by choosing several parallel lines, one
corresponding to each bin. If at any point breaker plays a point on
one of these lines then Maker views that line as `dead'. The ideas of
Lemma~\ref{l:solo-bin-ball} suggest that in time $T$ Maker should be
able to make a set of size roughly $m(T)\log
T/b(T)=\Omega(m(T))$. Thus, for some $\eps>0$ Maker should be able to
get $\eps n$ points by time $m(t)=(1-\eps)n$; i.e., Maker
wins.

There are two problems with this argument: the first is that
Lemma~\ref{l:solo-bin-ball} is only an upper bound and the second is
that in the weighted bin game Maker can place arbitrary weights in
bins, but in the n-in-a-row game he has to place an integer number of
points on each line.

A short calculation (which we do below) solves the first problem,
and a little care with the rounding solves the second.

\begin{proof}[Proof of Lower bound in Theorem~\ref{main-theorem-a>1}]
  Suppose that $\alpha>0$ is fixed, $m(t)=t^\alpha$ and $b(t)\le
  C\log t$ for some constant $C$. Fix $0<\eps<1/4$ to be chosen
  later 
  Let $t_1$ be minimal such that $m(t_1)>(1-\eps) n$. Let $r$
  be the largest power of 2 such that $m(t_1-r)>n/2$, and let
  $t_0=t_1-r$.  We prove that Maker can win just using his moves
  between times $t_0$ and $t_1$. During this period Breaker is playing
  at most $C\log n$ points each turn.

  In fact, we show the stronger statement: Maker can ensure that there
  is a line segment with at least $\eps n$ points after Breaker's go
  at the end of these $r$ timesteps while only playing $n/2$ points each
  turn and allowing breaker to play $C\log n$ (rather than $C\log t$)
  each turn. Obviously if Maker can do this then he can win the
  n-in-a-row-game on his next turn by playing all of his at least
  $(1-\eps)n$ points on this line segment.

  With a slight abuse of notation let $m= n/2$ and $b=C\log n$.
  Maker's strategy is as follows. He picks $rb+1$ parallel lines (not
  through any point that has already been played). At any time during
  the next $r$ timesteps we call a line live if it does not contain any of
  Breaker's points.

  During the first $r/2$ timesteps Maker places his
  $mr/2$ points as uniformly as possible on the $rb+1$ lines
  regardless of what Breaker does during these turns. Each line
  receives at least 
  \[\left\lfloor \frac{mr}{2(rb+1)}\right\rfloor\ge \frac{m}{4b} \]
  points where the inequality holds for $n$ sufficiently large.

  Now during these $r/2$ timesteps breaker has killed at most $rb/2$
  lines. If Breaker has killed less than this number then Maker
  arbitrarily designates some lines killed until there are exactly
  this many killed lines. Thus after this period there are
  exactly $rb/2+1$ live lines.

  Then in the next $r/4$ timesteps Maker places his $mr/4$ points as
  uniformly as possible in these $rb/2+1$ remaining live lines. Each line receives at least 
  \[
  \left \lfloor\frac{mr}{4(\frac{rb}{2}+1)} \right\rfloor\geq \frac{m}{4b}.
  \]

  This time Breaker has killed at most $rb/4$ lines and so at least
  $rb/4+1$ live lines remain. As before Maker artificially designates
  lines killed until exactly this many live lines remain.

  Maker repeats this $\log_2{r}$ times. At the end of these $r$
  timesteps the single remaining live line will have at least
  $\frac{m}{4b}\log_2{r}$ points on it.

  However, since $m= n/2$ and $b= C\log n$, and 
  \[r \ge
  \frac12\left(\left(\frac{3n}{4}\right)^{\frac{1}{\alpha}}-\left(\frac{n}{2}\right)^{\frac{1}{\alpha}})\right)=kn^{1/\alpha}\]
  for some constant $k$ depending only on $\alpha$.  Thus,
\[
\frac{m}{4b}\log_2{r} \ge \frac{n}{8C\log n}\log_2 (kn^{1/\alpha})=\frac{n}{8C\alpha\log 2}(1-o(1))
\]

Hence, if we set $\eps<\frac{1}{8C\alpha\log 2}$, by the end of
these $r$ timesteps Maker has placed at least $\eps n$ points in the
last line, and so Maker wins.\end{proof}

\begin{remark}
    This proof shows that, for any fixed $\alpha$, there
    is a constant $K_1$ such that if $b(t) \leq C \log n$ then Maker
    can win with constant $K_1 / C$. In contrast, in the proof of the
    upper bound of Theorem \ref{main-theorem-a>1}, we saw that,
    provided $\alpha\ge 1$, there exists $K_2$ such that if $b(t) \geq C
    \log n$ Breaker can prevent Maker from winning with constant $K_2
    / \sqrt{C}$. Thus, in this more precise formulation we do have a
    slight gap between our upper and lower bounds even in the case
    $\alpha\ge 1$.
\end{remark}
\subsection*{The directed and batched games for $\alpha < 1$} 
In this subsection we prove Theorem~\ref{t:subcritical-lowerbound}
giving the lower bound for directed-Breaker version of the game.

\begin{proof}[Proof of Theorem~\ref{t:subcritical-lowerbound}]
  In fact we
  show the stronger result: with $m$ and $b$ as in the statement of
  the theorem Maker can win the batched version of the
  directed-Breaker game.
	
  Indeed, Maker chooses to play until time $m(t)=n/2$; i.e., until
  time $t=(n/2)^{1/\alpha}$. During this time Maker plays
  $\Omega(n^{1+1/\alpha})$ points in total. He plays these as the
  integer points in a rectangle with sides $n$ and
  $\Omega(n^{1/\alpha})$. This sets contains
  $\Omega(n^{1/\alpha}n^{1/\alpha-1})=\Omega(n^{2/\alpha-1})$ lines
  containing $n$ of Maker's points ($\Omega(n^{1/\alpha})$ starting
  points and $\Omega(n^{1/\alpha-1})$ gradients).

  Breaker has at most  
  \[b(n^{1/\alpha})n^{1/\alpha}=o(n^{1/\alpha-1}n^{1/\alpha})=o(n^{2/\alpha-1})\]
  points to play so at least one of Maker's lines does not receive a
  point and, thus, Maker wins.
\end{proof}

We have seen in Theorem~\ref{t:subcritical-upperbound} that we can
match this lower bound for the directed-Breaker game. However, in the
ordinary batched game (i.e., not the directed-Breaker version) Breaker
can win whenever $b(t)=\omega(\log t)$.

\begin{proposition}
  Suppose that $\alpha<1$, $m(t)=t^\alpha$ and $b(t)=\omega(\log
  t)$. Then Breaker can win the batched game.
\end{proposition}
\begin{proof}
Suppose Maker chooses $T$ and $\eps>0$ such that $m(T) = (1-\eps)n$,
and has chosen his points. Note that $T^\alpha\le n$ and that $T\to
\infty$ as $n\to \infty$.  We show that Breaker can prevent Maker from
forming any line segment with $\eps T^\alpha\le \eps n$ points. We use
probabilistic methods to construct Breaker's set. Let $A$ be a subset
of Maker's points chosen independently at random with probability
$p=\frac{2\log T}{\eps T^\alpha}$.

  Maker plays at most $T^{1+\alpha}$ points so, by the
  Szemer\'edi-Trotter Theorem, Maker has at most
  \[
  O\left(\frac{(T^{1+\alpha})^2}{(\eps T^\alpha)^3} +\frac{T^{1+\alpha}}{\eps T^\alpha}\right)=O(T^{2-\alpha})
  \] (since $\alpha<1$) line segments with
  more than $\eps T^\alpha$ points in them. (As in Section \ref{s:elem} we
  are considering any line segments containing $\sim k \eps T^\alpha$ points
  as $\lceil k\rceil$ line segments each containing at most $\eps T^\alpha$
  points.)

  The probability that the set $A$ contains no point from a line
  segment of length $\eps T^\alpha$ is $(1-\frac{2\log T}{\eps
    T^\alpha})^{\eps T^\alpha}\approx T^{-2}$. Hence the probability
  that there exists such a line segment that does not receive a point from $A$
  is $O(T^{2-\alpha}\times T^{-2})=O(T^{-\alpha})$ which is less than
  $1/10$ for $n$, and thus $T$, sufficiently large.

  The probability that the set $A$ contains at most
  $pT^{1+\alpha}=\frac{2}{\eps}T\log T$ points is approximately
  $1/2$. Thus, with positive probability, the set $A$ contains at most
  $\frac{2}{\eps}T\log T$ points and contains a point from every one
  of Maker's line segments of length at least $\eps T^\alpha$. In
  particular, there exists a set $A'$ satisfying both these conditions.

  Now, Breaker gets to play $B=\sum_{t=1}^Tb(t)$ points. Since
  $b(t)>\frac{4}{\eps}\log t$ for all sufficiently large $t$ we see that
  $B>\frac{2}{\eps}T\log T$ for all sufficiently large $T$ (and thus
  for all sufficiently large $n$).

  Thus, Breaker's strategy is to pick the set $A'$ given above. This
  shows that he can stop Maker forming a line of length $\eps
  T^\alpha<\eps n$ and so Breaker wins.
\end{proof}

\begin{remark}
  We note that, by a similar argument to Proposition~\ref{p:batched},
   Breaker can win the batched game with $m(t) = t^\alpha$ and $b(t) = \omega(1)$ for all
   $\alpha \geq 1$.
\end{remark}
\section{Extensions of the results.}
Although we have stated and proved the results in $\Z^2$ they apply in
rather more generality.  Indeed, since the proofs of the upper bounds
only rely on the Szemer\'edi-Trotter Theorem they apply anywhere that
theorem holds.  In particular, they apply in higher dimensions (with
the winning sets being lines), in $\Z^d$, $\Q^d$ and $\R^d$, and in
cases where the winning sets are more general curves: for example
solutions to polynomial equations, or any other curves any two of
which only intersect in a bounded number of places. In particular if
$\alpha\ge 1$ then Breaker can win any of these games whenever
$b(t)=\omega(\log t)$.

\section{Open questions}
In the case where $\alpha \geq 1$ there is a clear
threshold between a Maker win and a Breaker win, however when $\alpha
< 1$ the bounds are still very far apart. Our key question is to find
the correct bound in this case. Define the threshold function
\[ \beta_c(\alpha)=\sup\{\beta\colon\text{The game with $m(t)=t^\alpha$ and $b(t)=t^\beta$ is a Maker win.}\}
\]
\begin{question}
  Find $\beta_c(\alpha)$ for $\alpha<1$.
\end{question}
Our lower bound for $b(t)$ in this case is only logarithmic so we do
not even know the answer to the following simpler question.
\begin{question}
  Is $\beta_c(\alpha)>0$ for any $\alpha$?
\end{question}

The form of our bounds seem to suggest that as $\alpha$ increases
Breaker can win with fewer and fewer points. Indeed our upper bound
for $b$ is monotone decreasing.  Perhaps the actual threshold is also monotonic?
\begin{question}
  Is $\beta_c(\alpha)$ monotone decreasing?
\end{question}

We saw in the previous section that the results generalise to many
other settings. In many senses the most natural setting for our result
is $\Q^2$ rather than $\Z^2$. We have seen that for Maker to win with
constant $\eps$ Maker must guarantee to have a line segment with at
least $\eps n$ points on it after Breaker has played.  In $\Q^2$ in
the case $m(t)=b(t)$ we do not even know that Maker can
\emph{ever} guarantee to have such a line segment.

\begin{question}
  Fix $n$ and suppose that the game is played on $\Q^2$ with
  $m(t)=b(t)=t$.  Can Maker guarantee to have a line segment
  containing $n$ points after Breaker's move?  If so, by what time can
  he guarantee to have such a line segment? 
\end{question}

Another extension where we don't know the answer is the
following: define Maker's winning set to be sets of $n$ points whose
convex hull contains none of Breaker's points. This obviously
generalises the $n$ points on a line with none of Breaker's points
between them that we have been considering in this paper.
\begin{question}
  Let Maker's winning set be sets of $n$ points in $\Z^d$ whose convex
  hull contains none of Breaker's points. What is the threshold for
  $b(t)$ when $\alpha=1$? In particular can Maker win this game with
  $m(t)=t$ and $b(t)=t^\eps$ for some $\eps>0$?
\end{question}

\bibliography{szemtrot}
\bibliographystyle{plain}

\end{document}